\documentclass[12pt]{amsart}
\usepackage{amsmath}
\usepackage{amssymb,latexsym,amsthm,amsfonts,amscd,pgf,enumerate,ragged2e}
\usepackage{tikz,amsmath}
\usepackage{subcaption}
\usepackage{graphicx}
\usepackage{float}
\usepackage{ragged2e}
\usepackage{cite}
\usepackage{color}
\usetikzlibrary{calc}

\usepackage[pdftex, colorlinks, linkcolor=red,citecolor=green]{hyperref}
\newtheorem{theorem}{Theorem}[section]
\newtheorem{lemma}[theorem]{Lemma}

\newtheorem{corollary}[theorem]{Corollary}
\newtheorem{remark}[theorem]{Remark}

\newtheorem{example}[theorem]{Example}
\newtheorem{conjecture}[theorem]{Conjecture}
\newtheorem{definition}[theorem]{Definition}
\newtheorem{discussion}[theorem]{Discussion}

\numberwithin{equation}{section}

\def\K{\dim_{\mathbb{K}}}
\def\N{\mathbb{N}}
\def\FK{\mathbb{K}}

\def\C{\mathbb{C}}
\def\M{\mathcal{M}}
\def\J{\mathcal{J}}

\def\1M{\M_G^{(1)}}
\def\Q{\widetilde Q}

\begin{document}

\baselineskip=15.5pt

\title[Standard monomials of $1$-skeleton ideals of multigraphs]{Standard monomials of $1$-skeleton ideals of multigraphs}

\author[A. Roy]{Amit Roy}

\address{NISER Bhubaneswar,
 Khurda, Pipli, Near Jatni, Odisha 752050, India.}

\email{amitiisermohali493@gmail.com}


\date{}
 
\begin{abstract}
Given a graph $G$ on the vertex set $\{0,1,\ldots,n\}$ with the root vertex $0$, Postnikov and Shapiro associated a monomial ideal $\M_G$ in the polynomial ring $R=\FK[x_1,\ldots,x_n]$ over a field $\FK$ such that $\K(R/\M_G)=\det\widetilde L_G$, where $\widetilde L_G$ is the truncated Laplacian of $G$. Dochtermann introduced the $1$-skeleton ideal $\M_G^{(1)}$ of $\M_G$ which satisfies the property that $\K(R/\M_G^{(1)})\ge\det\widetilde Q_G$, where $\widetilde Q_G$ is the truncated signless Laplacian of $G$. In this paper we characterize all subgraphs of the multigraph $K_{n+1}^{a,1}$, in particular all simple graphs $G$, such that $\K(R/\M_G^{(1)})=\det\widetilde Q_G$. Moreover, we give examples of subgraphs $G$ of the complete multigraph $K_{n+1}^{a,b}$, in which the equality $\K(R/\M_G^{(1)})=\det\widetilde Q_G$ holds. We also provide a conjecture on the structure of a general multigraph satisfying the above-mentioned equality.

\end{abstract}
\maketitle

\section{Introduction}
\label{introduction}

Let $G$ be a multigraph on the vertex set $V=\{0,1,\ldots,n\}=\{0\}\cup[n]$ with root $0$ and the set of edges $E(i,j)=E(j,i)$ between $i,j\in V$. The adjacency matrix of $G$ is given by $A(G)=[a_{i,j}]_{0\le i,j\le n}$, where $a_{i,j}=a_{j,i}=|E(i,j)|$ for $i\neq j$, and $a_{i,i}=0$ for all vertex $i,j$ in $G$. Therefore, $G$ is assumed to be undirected and loopless. For any nonempty subset $S$ of $[n]=\{1,2,\ldots,n\}$, let $d_S(i)=\sum_{j\notin S}a_{i,j}$ for $i\in S$. Thus $d_S(i)$ is the number of edges incident to the vertex $i$ whose other end points are not in $S$. In particular, $d_i:=d_{\{i\}}(i)$ is the degree of the vertex $i$ in $G$. For a graph $G$, the Laplacian $L_G$ and the signless Laplacian $Q_G$ are defined as 
\[
 L_G=D(G)-A(G)\quad\text{and}\quad Q_G=D(G)+A(G),
\]
where $D(G)=\operatorname{diag}[d_0,d_1,\ldots,d_n]$ is the diagonal matrix of order $n+1$. Let $R=\FK[x_1,\ldots,x_n]$ be the polynomial ring in $n$ variables over a field $\FK$. Given a multigraph $G$ on the vertex set $V$, we consider the monomial ideal
$
 \M_G=\left\langle m_S=\prod_{i\in S}x_i^{d_S(i)}\mid\emptyset\neq S\subseteq[n]\right\rangle
$
in $R$, called the $G$-parking function ideal or the graphical parking function ideal \cite{GH}. More generally, for a directed graph $G$, Postnikov and Shapiro introduced the Artinian monomial ideal $\M_G$ and studied their combinatorial and homological properties in \cite{PS}. The standard monomials $\mathbf{x}^{\mathbf{a}}=\prod_{i=1}^nx_i^{a_i}$ of $R/\M_G$ (i.e., $\mathbf{x}^{\mathbf{a}}\notin\M_G$) correspond to the $G$-parking functions $\mathbf{a}=(a_1,\ldots,a_n)\in\N^n$, a natural generalization of the classical parking functions introduced by Konheim and Weiss \cite{KW}. Moreover, $\K(R/\M_G)=\det\widetilde L_G$, where $\widetilde L_G$ is the truncated Laplacian obtained from $L_G$ by deleting the row and column corresponding to the root vertex. Thus by the Matrix-Tree theorem \cite[Theorem 5.6.8]{RS}, the number of $G$-parking functions equals the number of spanning trees of $G$. For directed graphs Chebikin and Pylyavskyy \cite{ChPy}, for simple graphs Perkinson, Yang and Yu \cite{PYY}, and for multigraphs Gaydarov and Hopkins \cite{GH} have provided an explicit bijection between the set of spanning trees of $G$ and the set of $G$-parking functions. The ideals $\M_G$ have connections to `chip firing' \cite{BaSh} and a discrete Riemann-Roch theory for graphs \cite{BaNo}. For instance, Baker and Norine reinterpreted the standard monomials of $\M_G$ as `$q$-reduced divisors' in \cite{BaNo}, where they proved a Riemann-Roch theorem for graphs analogous to the classical statement for Riemann surfaces. Motivated by certain constructions in `hereditary chip firing' models, Dochtermann \cite{Do} introduced the notion of $k$-skeleton ideals $\M_G^{(k)}$ for a simple graph $G$. However, as noted in \cite{Do} the definitions are easily extended for a multigraph.

For $0\le k\le n-1$, the ideals 
\[
 \M_G^{(k)}:=\left\langle m_S=\prod_{i\in S}x_i^{d_S(i)}\mid\emptyset\neq S\subseteq[n]\quad\text{and}\quad |S|\le k+1\right\rangle
\]
are by definition subideals of $\M_G$, where $G$ is a multigraph on the vertex set $\{0\}\cup[n]$. Note that for $k=n-1$, $\M_G^{(n-1)}=\M_G$. For a complete simple graph $G$ and for all positive integer $k$, the Betti numbers of $\M_G^{(k)}$ are determined in \cite{KLS}. Moreover, the standard monomials of $R/\M_G^{(k)}$ for such a $G$ are identified with certain kind of vector-parking functions in \cite{DoKi}. 

Analogous to the ideal $\M_G$, the number of standard monomials of $R/\1M$ has a determinantal interpretation. Let $a,b\ge 1$ be integers and $G$ be a multigraph with adjacency matrix $A(G)=[a_{i,j}]_{0\le i,j\le n}$, where $a_{0,i}=a$ and $a_{i,j}=b$ for $i,j\in[n]$; $i\neq j$. The graph $G$ described above is called a complete multigraph and is denoted by $K_{n+1}^{a,b}$. Dochtermann showed that for the complete simple graph $G=K_{n+1}:=K_{n+1}^{1,1}$, the equality $\K(R/\1M)=\det\Q_G$ holds, where $\Q_G$ is the truncated signless Laplacian of $G$ obtained from $Q_G$ by deleting the row and column corresponding to the root $0$. Moreover, based on some initial computations, he asked whether it is true that for any simple graph $G$, $\K(R/\1M)\ge\det\Q_G$. This question has been answered affirmatively in \cite{KLR}. In fact, it is shown that for any multigraph $G$, we have  $\K(R/\1M)\ge\det\Q_G$; thus providing a lower bound for the number of standard monomials of $R/\1M$. As mentioned earlier, the graph $K_{n+1}$ attains this lower bound. Under certain conditions a characterization of the equality $\K(R/\1M)=\det\Q_G$ for simple graphs $G$ is obtained in \cite{KLR}.

In this article, we consider the question of classifying all multigraphs $G$ which satisfy the property that $\K(R/\1M)=\det\Q_G$. For this, we first define a {\it $d$-fold product} $G_1*_dG_2$ between two graphs $G_1$ and $G_2$, where $d\ge 1$ and  $|V(G_1*G_2)|=|V(G_1)|+|V(G_2)|-1$ (see Definition \ref{product graph}). In Section \ref{section 2}, we prove the following result which gives a new family of multigraph $G$ that satisfies the equality $\K(R/\1M)=\det\Q_G$.

  (i) (Corollary \ref{new graphs})
   Let $G_i$ be a multigraph on the vertex set $\{0,1,\ldots,n_i\}$ obtained from a complete multigraph $K_{n_i+1}^{a_i,b_i}$ by removing some edges through the root $0$, where $1\leq i\leq r$. Suppose $n=\sum_{i=1}^rn_i$. Let $G=((\cdots(G_1*_{d_1}G_2)*_{d_2}\cdots *_{d_{r-2}} G_{r-1})*_{d_{r-1}}G_r)$ be the multigraph on the vertex set $\{0,1,\ldots,n\}$, where $d_1\geq d_2\geq\dots\geq d_{r-1}$ are nonnegative integers, $b_1\geq d_1$, and $b_i\geq d_{i-1}$ for $2\le i\le r$.  Then $\dim_{\FK}(R/\M_G^{(1)})=\det\widetilde Q_G$.
   
   In Section \ref{characterizing subgraphs} we define the {\it maximally connected subgraphs} of a graph $G$ and prove the following classification result for subgraphs of the multigraph $K_{n+1}^{a,1}$.
   
   (ii) (Theorem \ref{for simple graph having multiple rooted edges}) Let $G$ be a subgraph of the complete multigraph $K_{n+1}^{a,1}$ on the vertex set $\{0,1,\ldots ,n\}$. The graph $G$ satisfies $\K(R/\M_G^{(1)})=\det\widetilde Q_G$ if and only if each maximally connected subgraph $G_i$ of $G$ with $|V(G_i)|=n_i$, is obtained from a complete multigraph $K_{n_i+1}^{a_i,1}$ by deleting some edges through the root $0$.
   
   As an immediate corollary we obtain the following classification result for simple graphs.
   
   (iii) (Corrolary \ref{simple graph result})
    Let $G$ be a simple graph  on $n+1$ vertices $\{0,1,\ldots,n\}$. Then  $\K(R/\M_G^{(1)})=\det\widetilde Q_G$ holds if and only if each maximally connected subgraph $G_i$ of $G$ with $|V(G_i)|=n_i$, is obtained from a complete simple graph $K_{n_i+1}$ by deleting some edges through the root $0$.
 
 Finally, based on some computations we conjecture in Section \ref{last section} that if $G$ is a multigraph satisfying the equality $\K(R/\1M)=\det\Q_G$, then after a renumbering of the vertices, $G$ is the graph described as in Corollary \ref{new graphs} (see Conjecture \ref{conjecture}).
    
\section{Standard monomials of a family of multigraphs}\label{section 2}

In this section we consider a monomial ideal defined by certain symmetric matrices. Using this we provide a new family of multigraphs $G$ satisfying the equality $\K(R/\1M)=\det\Q_G$.

 Let $n\ge 1$ and $M_n(\N)$ be the set of $n\times n$ matrices over nonnegative integers $\N$. Let 
 \[
 \mathcal{G}_n=\{H=[h_{i,j}]\in M_n(\N)\mid H^t=H\,\text{and}\, h_{ii}\ge\max_{j\neq i} h_{i,j}\,\text{for}\, 1\le i\le n.\}
 \] 
 For $H=[h_{i,j}]\in\mathcal{G}_n$ with $\alpha_i=h_{i,i}$, we associate a monomial ideal 
\[
 \mathcal{J}_H=\left\langle x_t^{\alpha_t},x_i^{\alpha_i-h_{i,j}}x_j^{\alpha_j-h_{i,j}}\mid 1\le t\le n,\,1\le i< j\le n\right\rangle
\]
in the polynomial ring $R=\FK[x_1,\ldots,x_n]$. Note that if we consider the matrix $H=\widetilde Q_G$, then $\J_H=\M_{G}^{(1)}$. The ideal $\J_H$ was introduced in \cite{KLR} where the following lower bound for the number of standard monomials was given in terms of the determinant of the matrix $H$.
\begin{theorem}\textup{\cite[Theorem 3.3]{KLR}}\label{inequality for multigraph}
 Let $H\in\mathcal{G}_n$ be positive semidefinite and let $\J_H$ be the monomial ideal in the polynomial ring $R$ associated to $H$. Then 
 \[
  \K\left(\frac{R}{\J_H}\right)\ge \det H.
 \]
\end{theorem}

By the above theorem, for a multigraph $G$ we have $\K(R/\M_G^{(1)})\ge\det\Q_G$, since $\Q_G$ is a positive semidefinite matrix \cite{KLR}. Theorem \ref{inequality for multigraph} is one of the key ingredients to obtain our results for the ideals $\M_G^{(1)}$. In fact, we obtain all our results for the ideal $\J_H$, where $H\in\mathcal G_n$ and then specialize them to the $1$-skeleton ideals $\M_G^{(1)}$.

Consider the complete multigraph $K_{n+1}^{a,b}$ and the monomial ideal $\M_{K_{n+1}^{a,b}}^{(1)}$ in the polynomial ring $R=\FK[x_1,\ldots,x_n]$. Sometimes we write $R=R_n$ to indicate the number of variables in $R$. Observe that the matrix $\Q_{K_{n+1}^{a,b}}=[h_{i,j}]_{1\le i,j\le n}$, where $h_{i,i}=a+(n-1)b$ for $i\in[n]$ and $h_{i,j}=b$ for $i\neq j$. The matrix $\Q_{K_{n+1}^{a,b}}$ is in $\mathcal G_n$. It has been shown in \cite[Remark 2.7]{KLS} that $\K\left(R_n/\M_{K_{n+1}^{a,b}}^{(1)}\right)=\det\widetilde Q_{K_{n+1}^{a,b}}$. This result has been generalized in the following way.

\begin{theorem}\textup{\cite[Theorem 2.5]{KLR}}\label{RC}
Let $G$ be a multigraph on $V=\{0,1,\ldots ,n\}$ obtained from the complete multigraph $K_{n+1}^{a,b}$ by deleting some edges through the root $0$. Then
\begin{equation}\label{e2}
\dim_{\mathbb K}\left(\frac{R_n}{\M_G^{(1)}}\right)=\det\widetilde Q_G.
\end{equation}
\end{theorem}

  We now proceed to give a generalization of Theorem \ref{RC} in the context of the ideal $\J_H$ for $H\in\mathcal G_n$. 
  Recall that given a matrix $H\in \mathcal{G}_n$ we have associated the monomial ideal $\J_H\subseteq R_n$ such that $\J_H=\M_G^{(1)}$ when $H=\widetilde Q_G$. For the graph $G$ in Theorem \ref{RC}, the truncated signless Laplacian $\Q_G=[h_{i,j}]_{1\le i,j\le n}$, where $b\le h_{i,i}\le a+(n-1)b$ for $i\in[n]$ and $h_{i,j}=b$ for $i\neq j$. 
 
 In the calculations below the $i^{th}$ row and $j^{th}$ column of a matrix $H$ are denoted by $\mathcal R_i$ and $\mathcal{C}_j$, respectively.
 The elementary column operation $\mathcal{C}_j\pm(\mathcal{C}_{k_1}+\dots+\mathcal C_{k_r})$ on $H$ means the matrix $H$ is transformed to a matrix $H'$, where only $j^{th}$ column $\mathcal{C}_{j}'$ of $H'$ differs from the $j^{th}$ column $\mathcal{C}_j$ of $H$ and $\mathcal{C}_j'=\mathcal{C}_j\pm(\mathcal{C}_{k_1}+\dots+\mathcal C_{k_r})$. The elementary row operations are also defined in a similar way.

\begin{theorem}\label{rooted edge matrix version}
 Let $H=[h_{i,j}]_{1\le i,j\le n}$ 
be a matrix over nonnegative integers such that $h_{i,i}=a_i$ for $i\in[n]$ and $h_{i,j}=b$ for $i\ne j$, where $a_i\geq b$ for each $i\in[n]$. Then
\[\dim_{\mathbb{K}}\left(\frac{R_n}{\J_H}\right)=\det H.\]
\end{theorem}

\begin{proof}
  We prove this by induction on $n$. For $n=2$ we have that $H=
 \begin{bmatrix}
  a_1 & b \\
  b & a_2
 \end{bmatrix}_{2\times 2}
$ and the ideal $\J_H=\left\langle x_1^{a_1},x_2^{a_2}, x_1^{a_1-b}x_2^{a_2-b} \right\rangle$. Thus $\dim_{\mathbb{K}}(R_2/\J_H)=a_1a_2-b^2=\det H$.

Suppose $n\geq 3$ and the theorem is true for any $m$ with $m<n$. If $b=0$, then $\J_H=\left\langle x_i^{a_i}:1\le i\le n\right\rangle$. Thus $\K(R_n/\J_H)=\prod_{i=1}^na_i=\det H$. If $a_i=b$ for each $i$, then $\dim_{\mathbb{K}}(R_n/\J_H)=0=\det H$. Hence, without loss of generality, assume that $a_1>b>0$. Let $r=a_1-b>0$. The ideal $\J_H=\left\langle x_l^{a_l}, x_i^{a_i-b}x_j^{a_j-b}: 1\leq l\leq n,~1\leq i<j\leq n \right\rangle$.

Let $H_1=\operatorname{diag}[b,a_2-b,a_3-b,\ldots,a_n-b]$ be the $n\times n$ diagonal matrix and let $H_2$ be the matrix obtained from $H$ by deleting row $\mathcal{R}_1$ and column $\mathcal{C}_1$.
 We see that $\left(\J_H:x_1^r\right)=\J_{H_1}$ and $\left\langle \J_H,x_1^r\right\rangle=\langle \J_{H_2},x_1^r\rangle$. Therefore, $\dim_{\mathbb{K}}(R_n/\left(\J_H:x_1^r\right))=\dim_{\mathbb{K}}(R_n/\J_{H_1})=\det H_1$ (since $H_1$ is a diagonal matrix). Moreover, $\dim_{\mathbb{K}}(R_n/\left\langle \J_H,x_1^r\right\rangle)=r\dim_{\mathbb{K}}(R_{n-1}/\J_{H_2})=r\det H_2$ (by induction hypothesis). Consider the short exact sequence of $\mathbb{K}$-vector spaces,
\begin{align}\label{short exact sequence}
 0\rightarrow\frac{R_n}{\left(\J_H:x_{1}^{r}\right)}\xrightarrow{\mu_{x_{1}^{r}}}\frac{R_n}{\J_H}x\xrightarrow{\nu}\frac{R_n}{\left\langle \J_H,x_{1}^{r}\right\rangle}\rightarrow 0,
\end{align}
where $\mu_{x_{1}^{r}}$ is the map induced by multiplication by $x_1^r$ and $\nu$ is the natural quotient map. We have, $\dim_{\mathbb{K}}(R_n/\J_H)=\K(R_n/\left(\J_H:x_{1}^{r}\right))+\K(R_n/\left\langle \J_H,x_{1}^{r}\right\rangle)=\K(R_n/\J_{H_1})+r\K(R_{n-1}/\J_{H_2})$. Hence, \[\dim_{\mathbb{K}}(R_n/\J_H)=\det H_1+r\det H_2.\]

Writing $a_1=b+r$ and using the additivity property of the determinant, we see that $\det H=r\det H_2+\det A$, where $A$ is the matrix obtained from $H$ by replacing the element $a_1$ with $b$. On applying elementary column and row operations, $\mathcal{C}_2-\mathcal{C}_1, \mathcal{R}_2-\mathcal{R}_1,\ldots ,\mathcal{C}_n-\mathcal{C}_1, \mathcal{R}_n-\mathcal{R}_1$ on $A$, it reduces to the matrix $H_1$.
Thus $\det A=\det H_1$ and hence $\K(R_n/\J_H)=\det H$. \qed

\end{proof}
\vspace{4mm}
In the next theorem we further generalize Theorem \ref{rooted edge matrix version}. Consequently, we provide some new families of multigraphs $G$ which satisfy $\K(R/\M_G^{(1)})=\det\Q_G$.

\begin{theorem}\label{matrix version of join of graphs}
 Let $H_i=\begin{bmatrix}
           \alpha_{i,1} & b_i & \cdots & b_i \\
           b_i & \alpha_{i,2} & \cdots & b_i \\
           \vdots & \vdots & \ddots & \vdots \\
           b_i & b_i & \cdots & \alpha_{i,n_i}
          \end{bmatrix}_{n_i\times n_i}
$ with $\alpha_{i,j}\geq b_i$ and let $A_i$ be the $n_i\times n_i$ matrix with all entries equal to $d_i$. Consider the matrix 
\begin{align}\label{join graph matrix}
H=
\begin{bmatrix}
H_1 & A_1 & A_2 & \cdots & A_{r-1} \\
A_1 & H_2 & A_2 & \cdots & A_{r-1} \\
A_2 & A_2 & H_3 & \cdots & A_{r-1} \\
\vdots & \vdots & \vdots & \ddots & \vdots \\
A_{r-1} & A_{r-1} & A_{r-1} & \cdots & H_r
\end{bmatrix}_{\sum_{i=1}^rn_i\times\sum_{i=1}^rn_i}
\end{align}
with $b_1\geq d_1,\,b_i\geq d_{i-1}$ for $2\leq i\leq r$, and $d_i\geq d_{i+1}$ for $1\leq i\leq r-2$. Assume that $\alpha_{i,j},\,b_i$ and $d_i$ are all nonnegative integers. Suppose $n=\sum_{i=1}^rn_i$. Then $\K(R_n/\J_H)=\det H$.
\end{theorem}

\begin{proof}
 If $n=2$, then we see that $\K(R_2/\J_H)=\det H$. We prove the theorem by induction on the order $n$ of $H$. Assume that $n\geq 3$ and the theorem holds for every $m\times m$ matrix of the above form for $m<n$. The monomial ideal $\J_H$ is generated by the following monomials
 \begin{align*}
 x_{i,j}^{\alpha_{i,j}}&: 1\leq i\leq r,~1\leq j\leq n_i\,,\\
 x_{i,u}^{\alpha_{i,u}-b_i}x_{i,v}^{\alpha_{i,v}-b_i}&: 1\leq i\leq r,~1\leq u<v\leq n_i\,, \\
 x_{s,w}^{\alpha_{s,w}-d_{t-1}}x_{t,l}^{\alpha_{t,l}-d_{t-1}}&: 1\leq s\leq r-1,~1\leq w\leq n_s,~s+1\leq t\leq r,~1\leq l\leq n_t\,. 
 \end{align*} 
 
 We now divide the proof into two cases:
 
 \noindent
 {\bf Case I :} $n_r=1$. Let $B_1$ be the matrix obtained from $H$ by deleting the row and column containing the diagonal element $\alpha_{r1}$ and $B_2$ be an $(n-1)\times(n-1)$ matrix whose all entries are $d_{r-1}$. Let $B_3=B_1-B_2$. We see that the ideals $\left(\J_H:x_{r,1}^{\alpha_{r,1}-d_{r-1}}\right)=\left\langle \J_{B_3},x_{r,1}^{d_{r-1}}\right\rangle$ and $\left\langle \J_{H},x_{r,1}^{\alpha_{r,1}-d_{r-1}}\right\rangle=\left\langle \J_{B_1}, x_{r,1}^{\alpha_{r,1}-d_{r-1}}\right\rangle$. By the induction hypothesis,
 
 \begin{align*}
  \K\left(\frac{R_{n-1}}{\J_{B_3}}\right)
   =\det B_3 \quad \text{and}\quad \K\left(\frac{R_{n-1}}{\J_{B_1}}\right)=\det B_1.
 \end{align*}
Thus $\K\left(R_n/(\J_H:x_{r,1}^{\alpha_{r,1}-d_{r-1}})\right)=d_{r-1}\K(R_{n-1}/\J_{B_3})=d_{r-1}\det B_3$. Moreover, we see that $\K\left(R_n/\left\langle \J_{H},x_{r,1}^{\alpha_{r,1}-d_{r-1}}\right\rangle\right)=(\alpha_{r,1}-d_{r-1})\K(R_{n-1}/\J_{B_1})=(\alpha_{r,1}-d_{r-1})\det B_1$.
Using the short exact sequence of $\mathbb{K}$-vector spaces
\begin{align*}
 0\rightarrow\frac{R_n}{\left(\J_H:x_{r,1}^{\alpha_{r,1}-d_{r-1}}\right)}\xrightarrow{\mu_{x_{r,1}^{\alpha_{r,1}-d_{r-1}}}}\frac{R_n}{\J_H}\xrightarrow{\nu}\frac{R_n}{\left\langle \J_{H},x_{r,1}^{\alpha_{r,1}-d_{r-1}}\right\rangle}\rightarrow 0,
\end{align*}
we get $\K\left(R_n/\J_H\right)=d_{r-1}\det B_3+(\alpha_{r,1}-d_{r-1})\det B_1$. Here $\mu_{x_{r,1}^{\alpha_{r,1}-d_{r-1}}}$ is the map induced by multiplication by $x_{r,1}^{\alpha_{r,1}-d_{r-1}}$ and $\nu$ is the natural quotient map. As the determinant is linear on columns, writing $\alpha_{r,1}=(\alpha_{r,1}-d_{r-1})+d_{r-1}$, we have $\det H=(\alpha_{r,1}-d_{r-1})\det B_1+\det B_4$, where $B_4$ is the matrix obtained from $H$ by replacing the diagonal element $\alpha_{r,1}$ with $d_{r-1}$. Applying elementary column and row operations, $\mathcal{C}_1-\mathcal{C}_n,\mathcal{R}_1-\mathcal{R}_n,\ldots ,\mathcal{C}_{n-1}-\mathcal{C}_{n},\mathcal{R}_{n-1}-\mathcal{R}_n$ on $B_4$, we get $\det B_4=d_{r-1}\det B_3$. Consequently, $\K(R_n/\J_H)=\det H$.

\vspace{5mm}
\noindent
{\bf Case II :} $n_r\geq 2$. Let $B_5$ be the matrix obtained from $H$ by first deleting rows and columns containing the diagonal elements $\alpha_{r,1},\alpha_{r,2},\ldots ,\alpha_{r,n_r-1}$ and then replacing the diagonal element $\alpha_{r,n_r}$ with $b_r$. Hence, $B_5$ is an $(n+1-n_r)\times(n+1-n_r)$ matrix. Note that the ideal $\left( \J_H:x_{r,n_r}^{\alpha_{r,n_r}-b_r} \right)=\left\langle \J_{B_5},x_{r,1}^{\alpha_{r,1}-b_r},\ldots ,x_{r,n_r-1}^{\alpha_{r,n_r-1}-b_r} \right\rangle$. By the induction hypothesis, $\K\left(R_{n+1-n_r}/\J_{B_5}\right)=\det B_5$. Thus
\begin{align*}
 \K\left(\frac{R_n}{\left( \J_H:x_{r,n_r}^{\alpha_{r,n_r}-b_r} \right)}\right)&=\left(\prod_{i=1}^{n_r-1}(\alpha_{r,i}-b_r)\right)\K\left(\frac{R_{n+1-n_r}}{\J_{B_5}}\right)\\
 &=\left(\prod_{i=1}^{n_r-1}(\alpha_{r,i}-b_r)\right)\det B_5
\end{align*}

 Let $B_6$ be the $(n-1)\times(n-1)$ matrix obtained from $H$ by deleting the row and column containing the diagonal element $\alpha_{r,n_r}$. We see that $\left\langle \J_H,x_{r,n_r}^{\alpha_{r,n_r}-b_r} \right\rangle=\left\langle \J_{B_6},x_{r,n_r}^{\alpha_{rn_r}-b_r} \right\rangle$. Also, $\K(R_{n-1}/\J_{B_6})=\det B_6$ (by the induction hypothesis). Thus 
 \begin{align*}
 \K\left(\frac{R_n}{\left\langle \J_H,x_{r,n_r}^{\alpha_{r,n_r}-b_r}\right\rangle}\right)&=(\alpha_{r,n_r}-b_r)\K\left(\frac{R_{n-1}}{\J_{B_6}}\right)\\
 &=(\alpha_{r,n_r}-b_r)\det B_6.
 \end{align*}
 
 \noindent
 Now using the short exact sequence of $\mathbb{K}$-vector spaces
\begin{align*}
 0\rightarrow\frac{R_n}{\left(\J_H:x_{r,n_r}^{\alpha_{r,n_r}-b_r}\right)}\xrightarrow{\mu_{x_{r,n_r}^{\alpha_{r,n_r}-b_r}}}\frac{R_n}{\J_H}\xrightarrow{\nu}\frac{R_n}{\left\langle \J_{H},x_{r,n_r}^{\alpha_{r,n_r}-b_r}\right\rangle}\rightarrow 0,
\end{align*}
we get $\K(R_n/\J_H)=\left(\prod_{i=1}^{n_r-1}(\alpha_{r,i}-b_r)\right)\det B_5+(\alpha_{r,n_r}-b_r)\det B_6$. Here $\mu_{x_{r,n_r}^{\alpha_{r,n_r}-b_r}}$ is the map induced by multiplication by $x_{r,n_r}^{\alpha_{r,n_r}-b_r}$ and $\nu$ is the natural quotient map. Writing $\alpha_{r,n_r}=(\alpha_{r,n_r}-b_r)+b_r$ and using the additivity property of the determinant we see that $\det H=(\alpha_{r,n_r}-b_r)\det B_6+\det B_7$, where $B_7$ is the matrix obtained from $H$ by replacing the diagonal element $\alpha_{r,n_r}$ with $b_r$. Applying the elementary column and row operations, $\mathcal{C}_i-\mathcal{C}_n,\mathcal{R}_i-\mathcal{R}_n$ for $n-n_r< i\leq n-1$ on $B_7$, we get, $\det B_7=\left(\prod_{i=1}^{n_r-1}(\alpha_{r,i}-b_r)\right)\det B_5$. Consequently, $\K(R_n/\J_H)=\det H$. \qed
\end{proof}
\vspace{4mm}
The above theorem provides example of a matrix $H\in\mathcal{G}_n$ such that $\K(R_n/J_H)=\det H$. Moreover, using this result we can give some new families of multigraphs $G$ for which the number of standard monomials of $R/\M_G^{(1)}$ is same as the determinant of the truncated signless Laplacian $\Q_G$. These graphs are expressed as a certain {\it $d$-fold product} of some multigraphs which are obtained from complete multigraphs by removing some edges incident to the root. 
  \begin{definition}\label{product graph}
  Let $G_1$ be a multigraph on the vertex set $\{0,1,\ldots,n\}$ and let $G_2$ be a multigraph on the vertex set $\{0,1,\ldots ,m\}$. Let $d$ be a nonnegative integer. Define the graph $G_1*_{d}G_2$ on the vertex set $\{0,1,\ldots,n,n+1,\ldots ,n+m\}$ as follows. If $1\le i,j\leq n$ then the number of edges between $i$ and $j$ in $G$ is same as the number of edges between $i$ and $j$ in $G_1$. If $i,j\geq n+1$ then the number of edges between $i$ and $j$ in $G$ is same as the number of edges between $i-n$ and $j-n$ in $G_2$. For $1\le i\le n$ the number of edges between $0$ and $i$ in $G$ is same as the number of edges between $0$ and $i$ in $G_1$.
  For $j\ge n+1$ the number of edges between $0$ and $j$ in $G$ is same as the number of edges between $0$ and $j-n$ in $G_2$. For each $1\le i\leq n$ and $j\geq n+1$ the number edges between $i$ and $j$ is exactly $d$. 
  \end{definition}
  \begin{example}
  We give an example of the graph constructed above in Figure \ref{figure 123143}.
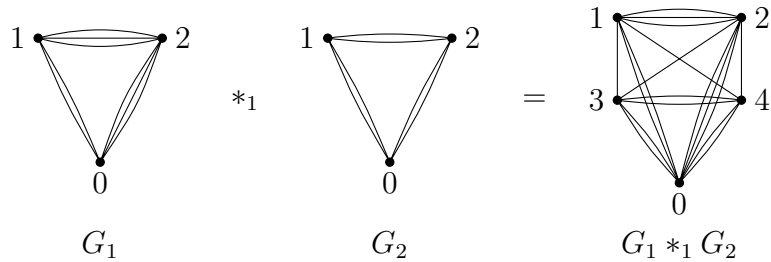
\begin{figure}[h!]
\tikzset{multi/.style={to path={
\pgfextra{%
 \pgfmathsetmacro{\startf}{-(#1-1)/2}  
 \pgfmathsetmacro{\endf}{-\startf} 
 \pgfmathsetmacro{\stepf}{\startf+1}}
 \ifnum 1=#1 -- (\tikztotarget)  \else
\foreach \i in {\startf,\stepf,...,\endf}
    {%
     (\tikztostart)        parabola[bend pos=0.5] bend +(0,0.2*\i)  (\tikztotarget)
      }
      \fi   
     \tikztonodes
      }}}
      
      \tikzset{multi1/.style={to path={
\pgfextra{%
 \pgfmathsetmacro{\startf}{-(#1-1)/2}  
 \pgfmathsetmacro{\endf}{-\startf} 
 \pgfmathsetmacro{\stepf}{\startf+1}}
 \ifnum 1=#1 -- (\tikztotarget)  \else
     let \p{mid}=($(\tikztostart)!0.5!(\tikztotarget)$) 
         in
\foreach \i in {\startf,\stepf,...,\endf}
    {%
     (\tikztostart) .. controls ($ (\p{mid})!\i*6pt!90:(\tikztotarget) $) .. (\tikztotarget)
      }
      \fi   
     \tikztonodes
}}}   

\begin{tikzpicture}
[scale=.55]
      \centering
\draw [fill] (0,1.5) circle [radius=0.1];
\draw [fill] (3,1.5) circle [radius=0.1];
\draw [fill] (1.5,-1.5) circle [radius=0.1];
\draw (0,1.5) edge[multi=3] (3,1.5);
\draw (0,1.5) edge[multi1=2] (1.5,-1.5);
\draw (3,1.5) edge[multi1=3] (1.5,-1.5);
\node at (5,0) {$*_1$};
\node at (1.5,-2) {$0$};
\node at (-0.5,1.5) {$1$};
\node at (3.5,1.5) {$2$};
\node at (1.5,-3.5) {$G_1$};
\draw [fill] (7,1.5) circle [radius=0.1];
\draw [fill] (10,1.5) circle [radius=0.1];
\draw [fill] (8.5,-1.5) circle [radius=0.1];
\draw (7,1.5) edge[multi=2] (10,1.5);
\draw (10,1.5) edge[multi1=2] (8.5,-1.5);
\node at (8.5,-2) {$0$};
\node at (10.5,1.5) {$2$};
\node at (6.5,1.5) {$1$};
\node at (8.5,-3.5) {$G_2$};
\draw (8.5,-1.5) edge[multi1=2] (7,1.5);
\node at (12,0) {$=$};
\draw [fill] (14,2) circle [radius=0.1];
\draw [fill] (17,2) circle [radius=0.1];
\draw [fill] (14,0) circle [radius=0.1];
\draw [fill] (17,0) circle [radius=0.1];
\draw [fill] (15.5,-2) circle [radius=0.1];
\node at (15.5,-2.5) {$0$};
\node at (13.5,2) {$1$};
\node at (17.5,2) {$2$};
\node at (13.5,0) {$3$};
\node at (17.5,0) {$4$};
\node at (15.5,-3.5) {$G_1*_1G_2$};
\draw (14,2) edge[multi=3] (17,2);
\draw (14,0) edge[multi=2] (17,0);
\draw (17,0) edge[multi1=2] (15.5,-2);
\draw (14,0) edge[multi1=2] (15.5,-2);
\draw (14,2)--(14,0)--(17,2)--(17,0)--(14,2);
\draw (15.5,-2) edge[multi1=3] (17,2);
\draw (15.5,-2) edge[multi1=2] (14,2);
\end{tikzpicture}
\caption{$G_1*_dG_2$ for $d=1$.}\label{figure 123143}
\end{figure}
  \end{example}
  \begin{corollary}\label{new graphs}
   Let $G_i$ be a multigraph on the vertex set $\{0,1,\ldots,n_i\}$ obtained from a complete multigraph $K_{n_i+1}^{a_i,b_i}$ by removing some edges through the root $0$, where $1\leq i\leq r$. Suppose $n=\sum_{i=1}^rn_i$. Let $G=((\cdots(G_1*_{d_1}G_2)*_{d_2}\cdots *_{d_{r-2}} G_{r-1})*_{d_{r-1}}G_r)$ be the multigraph on the vertex set $\{0,1,\ldots,n\}$, where $d_1\geq d_2\geq\dots\geq d_{r-1}$ are nonnegative integers, $b_1\geq d_1$, and $b_i\geq d_{i-1}$ for $2\le i\le r$.  Then $\dim_{\FK}(R_n/\M_G^{(1)})=\det\widetilde Q_G$.
  \end{corollary}
  \begin{proof}
   The truncated signless Laplacian $\widetilde Q_G$ of the graph $G$ is a matrix of the form \eqref{join graph matrix}. Taking $H=\widetilde Q_G$ in Theorem \ref{matrix version of join of graphs} we get our result. \qed
  \end{proof}
  
  \vspace{4mm}
  \noindent
  By the above corollary, we can deduce that $\K(R/\M_G^{(1)})=\det\Q_G$ for the graph $G=G_1*_1G_2$ in Figure \ref{figure 123143}.
  
  \section{Characterizing subgraphs of $K_{n+1}^{a,1}$}\label{characterizing subgraphs}

In this section we classify all subgraphs $G$ of the complete multigraph $K_{n+1}^{a,1}$ that satisfy $\K(R/\M_G^{(1)})=\det\widetilde Q_G$. In order to prove this we need the following results related to Hermitian matrices and their eigenvalues.

 Let $M\in M_n(\mathbb C)$ be a Hermitian matrix. The  eigenvalues of $M$ are arranged in a non-decreasing order $\lambda_1(M)\leq\lambda_2(M)\leq\cdots\leq\lambda_n(M)$. The Courant-Weyl inequalities \cite[Theorem 4.3.1]{HJ} compare eigenvalues of two Hermitian matrices with the eigenvalues of their sum.

    \begin{theorem}[Courant-Weyl]\label{CW}
    Let $M_1,M_2\in M_n({\mathbb C})$ be Hermitian matrices. Then
    \[
    \lambda_i(M_1+M_2)\leq\lambda_{i+j}(M_1)+\lambda_{n-j}(M_2)~{\rm for}~ j=0,1,\ldots ,n-i.
    \]
    \end{theorem}
    
The following application of the Courant-Weyl inequalities will be used in our proofs.

    \begin{lemma}\label{replacing an element}
    Let $M=(a_{i,j})_{n\times n}\in M_n(\C)$ be a Hermitian positive semidefinite matrix. Suppose $N$ is obtained from $M$ by replacing one diagonal element, say $a_{i,i}$, with a real number $b$ such that $a_{i,i}\geq b$. If $\det N>0$, then $N$ is a positive definite matrix.
    \end{lemma}
  
    \begin{proof}
    Let $\epsilon_{i,j}$ be the $n\times n$ matrix with $1$ at the $(i,j)^{th}$ place and zero elsewhere. Then $M=N+N'$, where $N'=(a_{i,i}-b)\epsilon_{i,i}$. Clearly, $\lambda_1(N')=\cdots=\lambda_{n-1}(N')=0$ and $\lambda_n(N')=a_{i,i}-b$. Since $M$ is positive semidefinite, $0\leq \lambda_1(M)\leq\cdots\leq\lambda_n(M)$. Taking $i=j=1$ in the Courant-Weyl inequalities for $M=N+N'$, we obtain $\lambda_1(M)\leq\lambda_2(N)+\lambda_{n-1}(N')=\lambda_2(N)$. Thus $0\leq \lambda_2(N)\leq\ldots\leq\lambda_n(N)$. As $\det N=\prod_{i=1}^n\lambda_i(N)>0$, $N$ must be positive definite. \qed
    \end{proof}
    \vspace{4mm}
    Given a block Hermitian matrix $M$ the Fischer's inequality provides an upper bound for $\det M$ in terms of the determinants of its diagonal blocks.

   \begin{theorem}[Fischer]\textup{\cite[Theorem 7.8.5]{HJ}}\label{FI}
    Let $M\in M_n(\mathbb C)$ be a positive semidefinite matrix having block decomposition 
    $M=
    \left(
    \begin{array}{c|c}
    A & B \\

    \hline

    B^* & C
    \end{array}
    \right)
    $ with square matrices $A$ and $C$. Then
    \begin{align}\label{definite to semidefinite}
    \det M\leq\det(A)\det(C).
    \end{align}    
    If $M$ is positive definite then equality occurs in \eqref{definite to semidefinite} if and only if $B=0$.
    \end{theorem}
   
  \vspace{4mm}
   The following Lemma will show us that in order to check whether a graph $G$ satisfy $\K(R/\M_G^{(1)})=\det\Q_G$, we only need to consider those graphs which have exactly two connected components after removing all the rooted-edges. We call these graphs as `maximally connected' graphs (see Definition \ref{defn essentially connected component}). 
\begin{lemma}\label{reducing to essential components}
 Let $G$ be a multigraph on $\{0,1,\ldots ,t\}$ and let $\widetilde G$ be the induced subgraph on the vertex set $[t]$. Suppose the connected components of $\widetilde G$ are $\widetilde G_1,\ldots ,\widetilde G_r$ with $|V_i|=t_i$, where $V_i=V(\widetilde G_i)$ for $1\leq i\leq r$. Consider the induced subgraphs $G_i=G_{V_i\cup\{0\}}$ of $G$. We have, $\K(R_t/\M_G^{(1)})=\det\widetilde Q_G$ if and only if $\K(R_{t_i}/\M_{G_i}^{(1)})=\det\widetilde Q_{G_i}$ for $1\le i\le r$ .
\end{lemma}

\begin{proof}
We observe that the $1$-skeleton ideal $\M_G^{(1)}=\sum_{i=1}^r\M_{G_i}^{(1)}$ such that

\[\K\left(\frac{R}{\M_{G}^{(1)}}\right)=\prod_{i=1}^r\K\left(\frac{R_{t_i}}{\M_{G_i}^{(1)}}\right).\]
Also, the truncated signless Laplacian $\widetilde Q_G$ is a block-diagonal matrix with each block corresponds to the truncated signless Laplacian of $G_i$. Therefore, $\det\widetilde Q_G=\prod_{i=1}^r\det\widetilde Q_{G_i}$. By Theorem \ref{inequality for multigraph}, $\K(R_{t_i}/\M_{G_i}^{(1)})\ge\det\widetilde Q_{G_i}$ for $1\le i\le r$. Moreover, $\K(R_{t_i}/\M_{G_i}^{(1)})\ge 1$ and $\det\widetilde Q_{G_i}\ge 0$ for $1\le i\le r$. Hence, $\K(R_t/\M_G^{(1)})=\det\widetilde Q_G$ if and only if $\K(R_{t_i}/\M_{G_i}^{(1)})=\det\widetilde Q_{G_i}$ for $1\le i\le r$. \qed
\end{proof}
\vspace{4mm}
We illustrate the above result by the following example.
\begin{example}
Let $G$ be the graph in Figure \ref{figure 123465} on the vertex set $\{0,1,\ldots,5\}$.
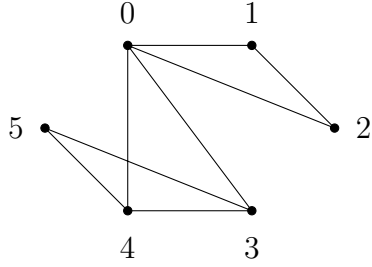
\begin{figure}[h!]
\centering
\begin{tikzpicture}
[scale=.55]
\draw [fill] (0,0) circle [radius=0.1];
\draw [fill] (3,0) circle [radius=0.1];
\draw [fill] (5,2) circle [radius=0.1];
\draw [fill] (3,4) circle [radius=0.1];
\draw [fill] (0,4) circle [radius=0.1];
\draw [fill] (-2,2) circle [radius=0.1];
\node at (0,-0.9) {$4$};
\node at (3,-0.9) {$3$};
\node at (5.7,2) {$2$};
\node at (3,4.8) {$1$};
\node at (0,4.8) {$0$};
\node at (-2.7,2) {$5$};
\draw (0,4)--(0,0)--(3,0)--(-2,2)--(0,0);
\draw (3,0)--(0,4)--(5,2)--(3,4)--(0,4);
\end{tikzpicture}\caption{A graph on vertex set $\{0\}\cup[5]$.}\label{figure 123465}
\end{figure}
  Let $V_1=\{1,2\}$ and $V_2=\{3,4,5\}$. Consider the induced subgraphs $G_i:=G_{V_i\cup\{0\}}$ for $i=1,2$. Identifying the vertex $i$ with the variable $x_i$ in the polynomial ring $R=\FK[x_1,\ldots,x_5]$ we see that $\M_G^{(1)}=\M_{G_1}^{(1)}+\M_{G_2}^{(1)}$, where $\M_{G_1}^{(1)}=\langle x_1^2,x_2^2,x_1x_2\rangle$ and $\M_{G_2}^{(1)}=\langle x_3^3,x_4^3,x_5^2,x_3^2x_4^2,x_3^2x_5,x_4^2x_5\rangle$. Moreover, $\K(R/\M_G^{(1)})=36=3\cdot 12=\K(R_1/\M_{G_1}^{(1)})\cdot\K(R_2/\M_{G_2}^{(1)})$, where $R_1=\FK[x_1,x_2]$ and $R_2=\FK[x_3,x_4,x_5]$. We also have $\K(R_i/\M_{G_i}^{(1)})=\det\Q_{G_i}$ for each $i=1,2$. Thus $\K(R/\M_{G}^{(1)})=\det\Q_{G}$.
  
\end{example}
The graphs $G_1$ and $G_2$ in above example are called {\it maximally connected subgraphs} of $G$. To be more precise, we have the following.
\begin{definition}\label{defn essentially connected component}
 Let $G$ be a multigraph on the vertex set $\{0,1,\ldots ,t\}$ and let $\widetilde{G}$ be the induced subgraph of $G$ on the vertex set $\{1,\ldots ,t\}$. If $\widetilde G_1$ is a connected component of $\widetilde G$ with vertex set $V_1$, then we call the induced subgraph $G_1=G_{V_1\cup\{0\}}$ of $G$ a maximally connected subgraph of $G$. Moreover, if $\widetilde{G}$ is connected, we say that $G$ is a maximally connected graph.
\end{definition}

Thus a graph is maximally connected if the only maximally connected subgraph of $G$ is itself.

\begin{remark}
A maximally connected graph may not be connected. For example, let $G$ be the simple graph on the vertex set $\{0,1,2\}$ obtained from $K_3$ by removing all the edges attached to the root $0$. Clearly, $G$ is maximally connected but not connected. Note that, a maximally connected graph $G$ is connected if and only if $G$ has at least one edge attached to the root $0$. 
\end{remark}
We now proceed to prove our main result in this section. For this we need to understand how the two quantities, determinant of a positive semidefinite matrix $H\in\mathcal{G}_n$ and the number of standard monomials of the ideal $\J_H$, are related.
\begin{discussion}\label{discussion 2.9}\normalfont
 Notice that by Lemma \ref{reducing to essential components}, to check whether for a multigraph $G$ the equality $\K(R/\M_G^{(1)})=\det\widetilde Q_G$ holds or not, it is enough to check this for its maximally connected subgraphs. Thus, without loss of generality we may assume that $G$ is a maximally connected multigraph on the vertex set $\{0,1,\ldots ,t\}$. Consider the induced subgraph $\widetilde G=G_{\{1,\ldots ,t\}}$.
Find two vertices in $\widetilde G$ such that the number of edges between them is maximum among the number of edges between any pair of vertices of $\widetilde G$. Rename these two vertices as $1$ and $2$. Suppose there are $b$ number of edges between them, where $b\ge 1$. Now find some $i\in V(\widetilde G)$, if exists, such that $i\notin\{1,2\}$ and there are $b$ number of edges from $i$ to both the vertices $1$ and $2$. Rename the vertex $i$ as $3$ and continue this way to find a maximal {\em clique} (an induced subgraph which is also a complete multigraph) on the vertex set (say) $\{1,2,\ldots ,n\}$ having $b$ number of edges between any two vertices. Then the truncated signless Laplacian of $G$ will be of the form
\begin{align}\label{graph to matrix}
 H:=\widetilde Q_G=\begin{bmatrix}
                    \alpha_1 & \cdots & b & b & d_{1,1} & d_{1,2} & \cdots & d_{1,m} \\
                    \vdots & \ddots & \vdots & \vdots & \vdots & \vdots & \ddots & \vdots \\
                    b & \cdots & \alpha_{n-1} & b & d_{n-1,1} & d_{n-1,2} & \cdots & d_{n-1,m} \\
                    b & \cdots & b & \alpha_n & d_{n,1} & d_{n,2} & \cdots & d_{n,m} \\
                    d_{1,1} & \cdots & d_{n-1,1} & d_{n,1} & \beta_1 & c_{1,2} & \cdots & c_{1,m} \\
                    d_{1,2} & \cdots & d_{n-1,2} & d_{n,2} & c_{1,2} & \beta_2 & \cdots & c_{2,m} \\
                    \vdots & \ddots & \vdots & \vdots & \vdots & \vdots & \ddots & \vdots \\
                    d_{1,m} & \cdots & d_{n-1,m} & d_{n,m} & c_{1,m} & c_{2,m} & \cdots & \beta_m
                   \end{bmatrix}_{t\times t},
\end{align}
where $n+m=t$ with $n\geq 2$. If $t=n$, then $\K(R_n/\M_G^{(1)})=\det\widetilde Q_G$ by Theorem \ref{rooted edge matrix version}. Therefore, we may assume that $m\geq 1$. Note that if $\alpha_i=\alpha_j=b$ for some $i\neq j$, then the two vertices $i$ and $j$ will form a connected component of $\widetilde G$ and hence $G$ will not be a maximally connected graph. Similarly, for $i\neq j$ we cannot have $\beta_i=c_{i,j}=\beta_j$ or $\alpha_i=d_{i,j}=\beta_j$. Also, if for some $1\leq j\leq m$, $d_{i,j}=b$ for each $1\leq i\leq n$, then the set of vertices $\{1,\ldots ,n\}$ will not form a maximal clique. Therfore, the matrix $H$ given in \eqref{graph to matrix} satisfies the following conditions.

\begin{align}
 &\text{For each}~ i\in[n]~\text{and}~j\in[m],~\alpha_i\geq b\geq 1,~ b\geq d_{i,j}\geq 0,~\beta_j\geq 1,~\text{and}~\beta_j\geq d_{i,j}.\label{equation 2.4a}\\[0.5em]
 &\text{For each}~ i,k\in[m],\text{and}~r<i<j,~\beta_i\geq c_{i,j},~\beta_i\geq c_{r,i}~\text{and}~b\geq c_{i,k}.\label{equation 2.4b}\\[0.5em]
 &\text{For each}~ i,j\in[n] ~\text{and}~ i\neq j,\,\text{either}~ \alpha_i>b ~\text{or}~ \alpha_j>b.\label{equation 2.4c}\\[0.5em]
 &\text{For each}~ i,j\in[m] ~\text{and}~ i\neq j,\,\text{either}~ \beta_i>c_{i,j} ~\text{or}~ \beta_j>c_{i,j}.\label{equation 2.4d}\\[0.5em]
 &\text{For each}~ i\in[n] ~\text{and}~ j\in[m],\,\text{either}~ \alpha_i>d_{i,j}~ \text{or}~ \beta_j>d_{i,j}.\label{equation 2.4e}\\[0.5em]
  &\text{For each}~j\in[m], ~\text{there exists some}~i\in[n] ~\text{such that}~ b\neq d_{i,j}.\label{equation 2.4f}
\end{align}
\end{discussion}

Thus we have the following lemma.

\begin{lemma}
Let $G$ be a maximally connected multigraph on the vertex set $\{0,1,\ldots,t\}$. If $G$ is not obtained from a complete multigraph by deleting only rooted edges then the signless Laplacian $\Q_G$ of $G$ is given by the equation \eqref{graph to matrix} and the entries of $\Q_G$ satisfies the equations \eqref{equation 2.4a} to \eqref{equation 2.4f}.
\end{lemma}
\begin{proof}
 The proof follows by proceeding as in Discussion \ref{discussion 2.9}.
\end{proof}

 Our aim is to prove the following theorem. The proof uses Fischer's inequality (Theorem \ref{FI}) and Theorem \ref{inequality for multigraph}.
\begin{theorem}\label{required for simple graphs}
 Consider the matrix $H$ given in \eqref{graph to matrix}. Assume that $H$ satisfies the conditions \eqref{equation 2.4a} to \eqref{equation 2.4f}. Suppose $H$ is a positive semidefinite matrix. If $\K(R_t/\J_H)=\det H$, then for each $1\leq l\leq m$ we have $d_{r,l}=d_{s,l}$ for $1\leq r,s\leq n$.
\end{theorem}

\begin{proof}
 For a contradiction assume that $d_{r,l}\neq d_{s,l}$ for some $l$. Without loss of generality, assume that $l=1$, i.e., $d_{r,1}\neq d_{s,1}$ for some $1\leq r<s\leq n$. We first show that
\begin{equation}\label{equation 2.5}
\alpha_i>b~\text{for all}~i\in[n].
\end{equation}
 If possible, let $\alpha_i=b$ for some $i\in[n]$. Without loss of generality, let $\alpha_n=b$. Then, $\alpha_i>b$ for each $i<n$ (by the condition \eqref{equation 2.4c}). Moreover, the ideal $\J_H$ is generated by the following monomials
 \begin{align*}
  x_i^{\alpha_i-b},x_n^b,y_j^{\beta_j}&:~1\leq i<n,~j\in[m], \\
  x_n^{b-d_{n,j}}y_j^{\beta_j-d_{n,j}}&:~j\in[m], \\
  y_i^{\beta_i-c_{i,j}}y_j^{\beta_j-c_{i,j}}&:~1\leq i<j\leq m.
 \end{align*}
 Consider the block-diagonal matrix $H_1=
\left(
\begin{array}{c|c}
D_{(n-1)\times (n-1)} &  \\

\hline

 & A_{(m+1)\times(m+1)}
\end{array}
\right)$,
 where the matrix $D=\operatorname{diag}[\alpha_1-b,\alpha_2-b,\ldots ,\alpha_{n-1}-b]$ is diagonal and the matrix $A$ is obtained from $H$ by deleting the rows and columns $\mathcal{R}_1, \mathcal{C}_1, \dots ,\mathcal{R}_{n-1}, \mathcal{C}_{n-1}$. We see that the ideal $\J_{H_1}=\J_H$. We also have $\K(R_t/\J_H)>0$ because of the conditions \eqref{equation 2.4a} and, \eqref{equation 2.4c} to \eqref{equation 2.4e}. Thus $\det H>0$ and hence, $H$ is a positive definite matrix. Applying elementary column and row operations $\mathcal{C}_1-\mathcal{C}_n,\mathcal{R}_1-\mathcal{R}_n,\ldots,\mathcal{C}_{n-1}-\mathcal{C}_n,\mathcal{R}_{n-1}-\mathcal{R}_n$ on $H$ we see that the reduced matrix
\begin{align*}
 \begin{bmatrix}
                    \alpha_1-b & \cdots & 0 & 0 & d_{1,1}-d_{n,1} & d_{1,2}-d_{n,2} & \cdots & d_{1,m}-d_{n,m} \\
                    \vdots & \ddots & \vdots & \vdots & \vdots & \vdots & \ddots & \vdots \\
                    0 & \cdots & \alpha_{n-1}-b & 0 & d_{n-1,1}-d_{n,1} & d_{n-1,2}-d_{n,2} & \cdots & d_{n-1,m}-d_{n,m} \\
                    0 & \cdots & 0 & b & d_{n,1} & d_{n,2} & \cdots & d_{n,m} \\
                    d_{1,1}-d_{n,1} & \cdots & d_{n-1,1}-d_{n,1} & d_{n,1} & \beta_1 & c_{1,2} & \cdots & c_{1,m} \\
                    d_{1,2}-d_{n,2} & \cdots & d_{n-1,2}-d_{n,2} & d_{n,2} & c_{1,2} & \beta_2 & \cdots & c_{2,m} \\
                    \vdots & \ddots & \vdots & \vdots & \vdots & \vdots & \ddots & \vdots \\
                    d_{1,m}-d_{n,m} & \cdots & d_{n-1,m}-d_{n,m} & d_{n,m} & c_{1,m} & c_{2,m} & \cdots & \beta_m
                   \end{bmatrix}
\end{align*}
is a positive definite matrix. By our assumption $d_{r,1}\neq d_{s,1}$ for some $1\leq r<s\leq n$, we see that $\det H_1>\det H$ by Fischer's inequality. The matrix $A$ is also positive definite, since $H$ is positive definite. Hence, $H_1$ is a positive definite matrix because $\alpha_i> b$ for each $i<n$. By Theorem \ref{inequality for multigraph}, $\K(R_t/\J_{H_1})\geq\det H_1$. Thus $\K(R_t/\J_H)=\K(R_t/\J_{H_1})\geq\det H_1>\det H$, a contradiction. Therefore, the matrix $H$ given in \eqref{graph to matrix} satisfies the condition \eqref{equation 2.5}.
  
Next we claim the following.
\begin{equation}\label{equation 2.6}
\hspace{-4.5cm}\text{For each}~j\in[m],~\text{either} ~\beta_j>d_{n,j}~\text{or}~b> d_{n,j}.
\end{equation}
  On the contrary, let $\beta_j=b=d_{n,j}$ for some $j\in[m]$. Without loss of generality, assume that $j=1$, i.e., $\beta_1=b=d_{n,1}$. Then, the ideal $\J_H$ is generated by the following monomials
\begin{align*}
 x_i^{\alpha_i},x_n^{\alpha_n-b},y_j^{\beta_j}:&~1\leq i< n,~j\in[m],\\
 x_i^{\alpha_i-b}x_j^{\alpha_j-b}:&~1\leq i<j< n,\\
 x_i^{\alpha_i-d_{i,j}}y_j^{\beta_j-d_{i,j}}:&~1\leq i< n,~j\in[m],\\
 y_i^{\beta_i-c_{i,j}}y_j^{\beta_j-c_{i,j}}:&~1\leq i<j\leq m.
\end{align*}
  Let $H_2$ be the matrix obtained from $H$ by replacing the diagonal element $\alpha_n$ with $\alpha_n-b$ and all other element of $\mathcal{R}_n$ and $\mathcal{C}_n$ with zero. 
   Then, $\J_H=\J_{H_2}$. We have $\K(R_t/\J_H)>0$ because of the conditions \eqref{equation 2.4a}, \eqref{equation 2.4c} to \eqref{equation 2.4e}. Thus $\det H>0$ and hence, $H$ is a positive definite matrix. Applying elementary column and row operations $\mathcal{C}_n-\mathcal{C}_{n+1}$ and $\mathcal{R}_n-\mathcal{R}_{n+1}$ on $H$ we see that the reduced matrix
  \begin{align*}
  \begin{bmatrix}
                    \alpha_1 & \cdots & b & b-d_{1,1} & d_{1,1} & d_{1,2} & \cdots & d_{1,m} \\
                    \vdots & \ddots & \vdots & \vdots & \vdots & \vdots & \ddots & \vdots \\
                    b & \cdots & \alpha_{n-1} & b-d_{n-1,1} & d_{n-1,1} & d_{n-1,2} & \cdots & d_{n-1,m} \\
                    b-d_{1,1} & \cdots & b-d_{n-1,1} & \alpha_n-b & 0 & d_{n,2}-c_{1,2} & \cdots & d_{n,m}-c_{1,m} \\
                    d_{1,1} & \cdots & d_{n-1,1} & 0 & b & c_{1,2} & \cdots & c_{1,m} \\
                    d_{1,2} & \cdots & d_{n-1,2} & d_{n,2}-c_{1,2} & c_{1,2} & \beta_2 & \cdots & c_{2,m} \\
                    \vdots & \ddots & \vdots & \vdots & \vdots & \vdots & \ddots & \vdots \\
                    d_{1,m} & \cdots & d_{n-1,m} & d_{n,m}-c_{1,m} & c_{1,m} & c_{2,m} & \cdots & \beta_m
                   \end{bmatrix}_{t\times t}
\end{align*}
  is a positive definite matrix. Because of the condition \eqref{equation 2.4f} we have $b\neq d_{i,1}$, for some $1\leq i<n$. Therefore, $\det H<\det H_2$ by Fischer's inequality. Since $\alpha_n> b$ and $H$ is a positive definite matrix, $H_2$ is also positive definite. By Theorem \ref{inequality for multigraph}, $\K(R_t/\J_{H_2})\geq\det H_2$. Thus $\K(R_t/\J_H)>\det H$, a contradiction. Hence, the matrix $H$ given in \eqref{graph to matrix} also satisfies condition \eqref{equation 2.6}.

  Now, $H$ is a positive semidefinite matrix given in \eqref{graph to matrix} which satisfies the conditions \eqref{equation 2.4a} to \eqref{equation 2.6}. The ideal $\J_H$ is generated by the monomials
\begin{align*}
 x_i^{\alpha_i},y_j^{\beta_j}&:~i\in[n],\, j\in[m],\\
 x_i^{\alpha_i-b}x_j^{\alpha_j-b}&:~1\leq i<j\leq n,\\
 x_i^{\alpha_i-d_{i,j}}y_j^{\beta_j-d_{i,j}}&:~i\in[n],\, j\in[m],\\
 y_i^{\beta_i-c_{i,j}}y_j^{\beta_j-c_{i,j}}&:~1\leq i<j\leq m.
\end{align*}
Consider the block-diagonal matrix $H_3=
\left(
\begin{array}{c|c}
\widehat D_{(n-1)\times (n-1)} &  \\

\hline

 & \widehat A_{(m+1)\times(m+1)}
\end{array}
\right)$, where the matrix $\widehat D=\operatorname{diag}[\alpha_1-b,\alpha_2-b,\ldots ,\alpha_{n-1}-b]$ is diagonal and the matrix $\widehat A$ is obtained from $H$ by deleting the rows and columns $\mathcal{R}_1, \mathcal{C}_1, \dots ,\mathcal{R}_{n-1}, \mathcal{C}_{n-1}$ and then replacing the element $\alpha_n$ with $b$.
  We see that $(\J_H:x_n^{\alpha_n-b})=\J_{H_3}$. Let $H_4$ be the matrix obtained from $H$ by replacing the diagonal element $\alpha_n$ with $\alpha_n-b$ and every other elements $\mathcal{R}_n$ and $\mathcal{C}_n$ with zero, i.e., 
   \begin{align*}
  H_4=\begin{bmatrix}
                     \alpha_1 & \cdots & b & 0 & d_{1,1} & d_{1,2} & \cdots & d_{1,m} \\
                     \vdots & \ddots & \vdots & \vdots & \vdots & \vdots & \ddots & \vdots \\
                     b & \cdots & \alpha_{n-1} & 0 & d_{n-1,1} & d_{n-1,2} & \cdots & d_{n-1,m} \\
                     0 & \cdots & 0 & \alpha_n-b & 0 & 0 & \cdots & 0 \\
                     d_{1,1} & \cdots & d_{n-1,1} & 0 & \beta_1 & c_{1,2} & \cdots & c_{1,m} \\
                     d_{1,2} & \cdots & d_{n-1,2} & 0 & c_{1,2} & \beta_2 & \cdots & c_{2,m} \\
                     \vdots & \ddots & \vdots & \vdots & \vdots & \vdots & \ddots & \vdots \\
                     d_{1,m} & \cdots & d_{n-1,m} & 0 & c_{1,m} & c_{2,m} & \cdots & \beta_m
                    \end{bmatrix}_{t\times t}.
 \end{align*}
  We have $\J_{H_4}=\left\langle \J_H,x_n^{\alpha_n-b} \right\rangle$. Now, from the short exact sequence of $\mathbb{K}$-vector spaces,
 \begin{align*}
 0\rightarrow\frac{R_t}{\left(\J_H:x_{n}^{\alpha_n-b}\right)}\xrightarrow{\mu_{x_{n}^{\alpha_n-b}}}\frac{R_t}{\J_H}\xrightarrow{\nu}\frac{R_t}{\left\langle \J_H,x_{n}^{\alpha_n-b}\right\rangle}\rightarrow 0,
\end{align*}
  we have $\K(R_t/\J_H)=\K(R_t/\J_{H_3})+\K(R_t/\J_{H_4})$. Here $\mu_{x_{n}^{\alpha_n-b}}$ is the map induced by multiplication by $x_{n}^{\alpha_n-b}$ and $\nu$ is the natural quotient map. Now writing $\alpha_n=(\alpha_n-b)+b$ and using the additivity property of the determinant, we get $\det H=\det H_4+\det B$, where $B$ is the matrix $H$ except the element $\alpha_n$ is replaced with $b$. The matrix $H_4$ is positive semidefinite since $H$ is positive semidefinite and $\alpha_n> b$. By Theorem \ref{inequality for multigraph}, $\K(R_t/\J_{H_4})\geq\det H_4$. We also have $\K(R_t/\J_{H_3})>0$ because of the conditions \eqref{equation 2.4a}, \eqref{equation 2.4d}, \eqref{equation 2.5} and \eqref{equation 2.6}. Since $\K(R_t/\J_H)=\det H$ we must have $\det B>0$. By Lemma \ref{replacing an element}, $B$ is a positive definite matrix since $\alpha_n\geq b$ and $H$ is a positive semidefinite matrix. Applying elementary column and row operations $\mathcal{C}_1-\mathcal{C}_n,\mathcal{R}_1-\mathcal{R}_n,\ldots,\mathcal{C}_{n-1}-\mathcal{C}_n,\mathcal{R}_{n-1}-\mathcal{R}_n$ on $B$ we see that the reduced matrix
  \begin{align*}
 \begin{bmatrix}
                    \alpha_1-b & \cdots & 0 & 0 & d_{1,1}-d_{n,1} & d_{1,2}-d_{n,2} & \cdots & d_{1,m}-d_{n,m} \\
                    \vdots & \ddots & \vdots & \vdots & \vdots & \vdots & \ddots & \vdots \\
                    0 & \cdots & \alpha_{n-1}-b & 0 & d_{n-1,1}-d_{n,1} & d_{n-1,2}-d_{n,2} & \cdots & d_{n-1,m}-d_{n,m} \\
                    0 & \cdots & 0 & b & d_{n,1} & d_{n,2} & \cdots & d_{n,m} \\
                    d_{1,1}-d_{n,1} & \cdots & d_{n-1,1}-d_{n,1} & d_{n,1} & \beta_1 & c_{1,2} & \cdots & c_{1,m} \\
                    d_{1,2}-d_{n,2} & \cdots & d_{n-1,2}-d_{n,2} & d_{n,2} & c_{1,2} & \beta_2 & \cdots & c_{2,m} \\
                    \vdots & \ddots & \vdots & \vdots & \vdots & \vdots & \ddots & \vdots \\
                    d_{1,m}-d_{n,m} & \cdots & d_{n-1,m}-d_{n,m} & d_{n,m} & c_{1,m} & c_{2,m} & \cdots & \beta_m
                   \end{bmatrix}
\end{align*}
is a positive definite matrix. Since $d_{r,1}\neq d_{s,1}$ for some $1\leq r<s\leq n$, we have $\det B<\det H_3$, by Fischer's inequality. The matrix $H_3$ is positive definite since $B$ is a positive definite matrix and $\alpha_i> b$ for $1\leq i<n$. By Theorem \ref{inequality for multigraph}, $\K(R_t/\J_{H_3})\geq\det H_3$. Thus $\K(R_t/\J_H)\geq\det H_3+\det H_4>\det B+\det H_4=\det H$, a contradiction, and this proves the theorem. \qed
\end{proof}

\vspace{4mm}
  In view of Theorem \ref{RC}, if $G$ is a simple graph on the vertex set $\{0,1,\ldots ,n\}$, obtained from a complete simple graph $K_{n+1}$ by deleting some edges through the root $0$, then $\K(R_n/\M_G^{(1)})=\det\widetilde Q_G$. Furthermore, by Lemma \ref{reducing to essential components}, in order to check for a graph $G$ when the equality $\K(R_n/\M_G^{(1)})=\det\widetilde Q_G$ holds we just need to check for its maximally connected subgraphs. Now as a consequence of Theorem \ref{required for simple graphs} we can characterize all simple graphs $G$ which satisfy the property $\K(R_n/\M_{G}^{(1)})=\det\widetilde Q_G$. More generally, we prove the following:
  
  \begin{theorem}\label{for simple graph having multiple rooted edges}
  Let $G$ be a subgraph of the complete multigraph $K_{n+1}^{a,1}$ on the vertex set $\{0,1,\ldots ,n\}$. The graph $G$ satisfies $\K(R_n/\M_G^{(1)})=\det\widetilde Q_G$ if and only if each maximally connected subgraph $G_i$ of $G$ with $|V(G_i)|=n_i$, is obtained from a complete multigraph $K_{n_i+1}^{a_i,1}$ by deleting some edges through the root $0$.
  \end{theorem}
  
  \begin{proof}
    First suppose that $G$ is a subgraph of the complete multigraph $K_{n+1}^{a,1}$ such that $G$ is maximally connected and $\K(R_n/\M_G^{(1)})=\det\widetilde Q_G$. We proceed in the same line as in Discussion \ref{discussion 2.9}. Suppose $\widetilde G$ is the induced subgraph $G_{\{1,\ldots ,n\}}$ of $G$. We find two vertices in $\widetilde G$ such that there is an edge between them and rename these two vertices as $1$ and $2$. Now find another vertex (if it exists) which has edges connecting both $1$ and $2$. Rename the new vertex as $3$ and continue this way to find a maximal clique on the vertex set say $\{1,2,\ldots ,r\}$ for $r\leq n$. The truncated signless Laplacian of $G$ will be of the form
    
    \begin{align*}
 \widetilde Q_G=\begin{bmatrix}
                    \alpha_1 & \cdots & 1 & 1 & d_{1,1} & d_{1,2} & \cdots & d_{1,m} \\
                    \vdots & \ddots & \vdots & \vdots & \vdots & \vdots & \ddots & \vdots \\
                    1 & \cdots & \alpha_{r-1} & 1 & d_{r-1,1} & d_{r-1,2} & \cdots & d_{r-1,m} \\
                    1 & \cdots & 1 & \alpha_r & d_{r,1} & d_{r,2} & \cdots & d_{r,m} \\
                    d_{1,1} & \cdots & d_{r-1,1} & d_{r,1} & \beta_1 & c_{1,2} & \cdots & c_{1,m} \\
                    d_{1,2} & \cdots & d_{r-1,2} & d_{r,2} & c_{1,2} & \beta_2 & \cdots & c_{2,m} \\
                    \vdots & \ddots & \vdots & \vdots & \vdots & \vdots & \ddots & \vdots \\
                    d_{1,m} & \cdots & d_{r-1,m} & d_{r,m} & c_{1,m} & c_{2,m} & \cdots & \beta_m
                   \end{bmatrix}_{n\times n},
\end{align*}
  where $r+m=n$, and for each $1\leq j\leq m$ there exists some $i\in[r]$ such that $d_{i,j}=0$. Moreover, the entries of $\Q_G$ satisfies the equations \eqref{equation 2.4a} to \eqref{equation 2.4f}. Since $\K(R_n/\M_G^{(1)})=\det\widetilde Q_G$, we must have $d_{i,j}=0$ for all $i$ and $j$, by Theorem \ref{required for simple graphs}. Hence, we have $r=n$ since $G$ is maximally connected. Thus $G$ is obtained from a complete multigraph $K_{n+1}^{a,1}$ by deleting some edges through the root $0$.

  If $G$ is a subgraph of $K_{n+1}^{a,1}$ and $\K(R_n/\M_G^{(1)})=\det\widetilde Q_G$, then by Lemma \ref{reducing to essential components} and by the above discussion we see that each maximally connected subgraph $G_i$ of $G$ is obtained from a complete multigraph $K_{n_i+1}^{a_i,1}$ by deleting some edges through the root $0$. 
  
  The converse follows from Lemma \ref{reducing to essential components} and Theorem \ref{rooted edge matrix version}. \qed
  \end{proof}
  \vspace{4mm}
  
  Thus for simple graphs we have the following result:
  \begin{corollary}\label{simple graph result}
    Let $G$ be a simple graph  on $n+1$ vertices $\{0,1,\ldots,n\}$. For the graph $G$, $\K(R_n/\M_G^{(1)})=\det\widetilde Q_G$ holds if and only if each maximally connected subgraph $G_i$ of $G$ with $|V(G_i)|=n_i$, is obtained from a complete simple graph $K_{n_i+1}$ by deleting some edges through the root $0$.
  \end{corollary}
  
  \begin{proof}
   Taking $a=1$ in Theorem \ref{for simple graph having multiple rooted edges} we get our result. \qed
  \end{proof}

 \section{Concluding remarks}\label{last section}

It would be interesting to characterize all multigraphs $G$ such that $\K(R/\1M)=\det\Q_G$. In this section we make some observations on the truncated signless Laplacians of such graphs. Moreover, based on some initial calculations we propose a conjecture on the structure of multigraphs $G$ such that the number of standard monomials of the $1$-skeleton ideal is same as the determinant of the truncated signless Laplacian of $G$.

Let $G$ be a multigraph on the vertex set $\{0,1,\ldots,t\}$. By Discussion \ref{discussion 2.9}, after a renumbering of the vertices the truncated signless Laplacian of $G$ is as in Equation (\ref{graph to matrix}), where the vertices $\{1,2,\ldots,n\}$ form a maximal clique in $G$. Then by Theorem \ref{required for simple graphs}, $d_{r,l}=d_{s,l}$ for $1\le r,s\le n$. Thus $\Q_G$ is of the form:

\begin{align*}\label{graph to matrix1}
 H:=\widetilde Q_G=\begin{bmatrix}
                    \alpha_1 & \cdots & b & b & d_{1} & d_{2} & \cdots & d_{m} \\
                    \vdots & \ddots & \vdots & \vdots & \vdots & \vdots & \ddots & \vdots \\
                    b & \cdots & \alpha_{n-1} & b & d_{1} & d_{2} & \cdots & d_{m} \\
                    b & \cdots & b & \alpha_n & d_{1} & d_{2} & \cdots & d_{m} \\
                    d_{1} & \cdots & d_{1} & d_{1} & \beta_1 & c_{1,2} & \cdots & c_{1,m} \\
                    d_{2} & \cdots & d_{2} & d_{2} & c_{1,2} & \beta_2 & \cdots & c_{2,m} \\
                    \vdots & \ddots & \vdots & \vdots & \vdots & \vdots & \ddots & \vdots \\
                    d_{m} & \cdots & d_{m} & d_{m} & c_{1,m} & c_{2,m} & \cdots & \beta_m
                   \end{bmatrix}_{t\times t},
\end{align*}
where $n+m=t$.
Among the vertices $\{n+1,\ldots,t\}$, we can again choose some vertices to form a maximal clique. Then as the following example suggests the number of edges joining these vertices and the vertices among $\{1,2,\ldots,n\}$ should be always same. For example, we take $n=2$ and $t=4$.
\begin{figure}[h!]
\tikzset{multi/.style={to path={
\pgfextra{%
 \pgfmathsetmacro{\startf}{-(#1-1)/2}  
 \pgfmathsetmacro{\endf}{-\startf} 
 \pgfmathsetmacro{\stepf}{\startf+1}}
 \ifnum 1=#1 -- (\tikztotarget)  \else
\foreach \i in {\startf,\stepf,...,\endf}
    {%
     (\tikztostart)        parabola[bend pos=0.5] bend +(0,0.2*\i)  (\tikztotarget)
      }
      \fi   
     \tikztonodes
      }}}
      
      \tikzset{multi1/.style={to path={
\pgfextra{%
 \pgfmathsetmacro{\startf}{-(#1-1)/2}  
 \pgfmathsetmacro{\endf}{-\startf} 
 \pgfmathsetmacro{\stepf}{\startf+1}}
 \ifnum 1=#1 -- (\tikztotarget)  \else
     let \p{mid}=($(\tikztostart)!0.5!(\tikztotarget)$) 
         in
\foreach \i in {\startf,\stepf,...,\endf}
    {%
     (\tikztostart) .. controls ($ (\p{mid})!\i*6pt!90:(\tikztotarget) $) .. (\tikztotarget)
      }
      \fi   
     \tikztonodes
}}}   

\begin{tikzpicture}
[scale=.55]
      \centering
\draw [fill] (5,2) circle [radius=0.1];
\draw [fill] (6.5,-1.8) circle [radius=0.1];
\draw [fill] (8,2) circle [radius=0.1];
\draw [fill] (5,0) circle [radius=0.1];
\draw [fill] (8,0) circle [radius=0.1];
\node at (4.5,2) {$1$};
\node at (8.5,2) {$2$};
\node at (4.5,0) {$3$};
\node at (8.5,0) {$4$};
\draw (5,2) edge[multi=2] (8,2);
\draw (5,0) edge[multi1=1] (8,2);
\draw (5,0) edge[multi=2] (8,0);
\draw (5,2)--(5,0);
\draw (8,0)--(6.5,-1.8)--(5,2);
\draw (8,2)--(8,0)--(5,2);
\node at (6.5,-2.5) {$0$};
\node at (6.5,-3.9) {$G$};

\draw [fill] (14,2) circle [radius=0.1];
\draw [fill] (15.5,-1.8) circle [radius=0.1];
\draw [fill] (17,2) circle [radius=0.1];
\draw [fill] (14,0) circle [radius=0.1];
\draw [fill] (17,0) circle [radius=0.1];
\node at (13.5,2) {$1$};
\node at (17.5,2) {$2$};
\node at (13.5,0) {$3$};
\node at (17.5,0) {$4$};
\draw (14,2) edge[multi=2] (17,2);
\draw (14,0) edge[multi1=2] (14,2);
\draw (14,0) edge[multi1=2] (17,2);
\draw (14,0) edge[multi=2] (17,0);
\draw (17,2)--(17,0)--(14,2);
\draw (17,0)--(15.5,-1.8)--(14,2);
\node at (15.5,-2.5) {$0$};
\node at (15.5,-3.9) {$H$};
\end{tikzpicture}\caption{}\label{figure 13}
\end{figure}
For the graphs $G$ and $H$ in Figure \ref{figure 13}, the vertices $\{1,2\}$ form a maximal clique among the non-root vertices and $\{3,4\}$ form a maximal clique among the remaining non-root vertices. Note that for the graph $G$ each vertices $3$ and $4$ are connected to the vertices $1$ and $2$ with the same number of edges. But this is not the case for $H$. A calculation with Macaulay2 \cite{GS} shows that $\K(R/\1M)=\det\Q_G=231$ but $\K(R/\M_H^{(1)})=564>550=\det\Q_H$. The graph $G$ can be written as $G'*_1G''$, where $G'$ and $G''$ are the induced subgraphs on the vertex sets $\{0,1,2\}$ and $\{0,3,4\}$, respectively. However, if we take $G_1=G'*_3G''$, then $\K(R/\M_{G_1}^{(1)}=3180>3103=\det\Q_{G_1}$. Here both $G'$ and $G''$ are obtained from the complete muligraph $K_3^{1,2}$ by deleting some rooted edges.

\begin{figure}[h!]
\tikzset{multi/.style={to path={
\pgfextra{%
 \pgfmathsetmacro{\startf}{-(#1-1)/2}  
 \pgfmathsetmacro{\endf}{-\startf} 
 \pgfmathsetmacro{\stepf}{\startf+1}}
 \ifnum 1=#1 -- (\tikztotarget)  \else
\foreach \i in {\startf,\stepf,...,\endf}
    {%
     (\tikztostart)        parabola[bend pos=0.5] bend +(0,0.2*\i)  (\tikztotarget)
      }
      \fi   
     \tikztonodes
      }}}
      
      \tikzset{multi1/.style={to path={
\pgfextra{%
 \pgfmathsetmacro{\startf}{-(#1-1)/2}  
 \pgfmathsetmacro{\endf}{-\startf} 
 \pgfmathsetmacro{\stepf}{\startf+1}}
 \ifnum 1=#1 -- (\tikztotarget)  \else
     let \p{mid}=($(\tikztostart)!0.5!(\tikztotarget)$) 
         in
\foreach \i in {\startf,\stepf,...,\endf}
    {%
     (\tikztostart) .. controls ($ (\p{mid})!\i*6pt!90:(\tikztotarget) $) .. (\tikztotarget)
      }
      \fi   
     \tikztonodes
}}}   

\begin{tikzpicture}
[scale=.55]
      \centering
\draw [fill] (0,0) circle [radius=0.1];
\draw [fill] (-3,2) circle [radius=0.1];
\draw [fill] (3,2) circle [radius=0.1];
\draw [fill] (3,4) circle [radius=0.1];
\draw [fill] (1,6) circle [radius=0.1];
\draw [fill] (-1,6) circle [radius=0.1];
\draw [fill] (-3,4) circle [radius=0.1];
\node at (0,-0.8) {$0$};
\node at (-3.8,2) {$1$};
\node at (-3.8,4) {$6$};
\node at (-1.8,6.2) {$5$};
\node at (1.8,6.2) {$4$};
\node at (3.8,4) {$3$};
\node at (3.8,2) {$2$};
\node at (2,0.3) {$H_1$};
\draw (0,0) edge[multi1=1] (-3,2);
\draw (-3,2) edge[multi1=2] (-3,4);
\draw (-3,2) edge[multi1=2] (-1,6);
\draw (-3,2) edge[multi1=2] (3,2);
\draw (3,2) edge[multi1=2] (-1,6);
\draw (3,4) edge[multi1=2] (1,6);
\draw (-3,4) edge[multi1=2] (3,2);
\draw (0,0)--(1,6);
\draw (-1,6)--(-3,4)--(3,4)--(-3,2)--(1,6)--(-3,4);
\draw (1,6)--(3,2)--(3,4)--(-1,6)--(1,6);

\draw [fill] (13,0) circle [radius=0.1];
\draw [fill] (10,2) circle [radius=0.1];
\draw [fill] (16,2) circle [radius=0.1];
\draw [fill] (16,4) circle [radius=0.1];
\draw [fill] (14,6) circle [radius=0.1];
\draw [fill] (12,6) circle [radius=0.1];
\draw [fill] (10,4) circle [radius=0.1];
\node at (13,-0.8) {$0$};
\node at (9.2,2) {$1$};
\node at (9.2,4) {$6$};
\node at (11.2,6.2) {$5$};
\node at (14.8,6.2) {$4$};
\node at (16.8,4) {$3$};
\node at (16.8,2) {$2$};
\node at (10.8,0.4) {$G_1$};
\draw (13,0) edge[multi1=1] (10,2);
\draw (10,2) edge[multi1=1] (10,4);
\draw (10,2) edge[multi1=1] (12,6);
\draw (10,2) edge[multi1=2] (16,2);
\draw (16,2) edge[multi1=1] (12,6);
\draw (16,4) edge[multi1=2] (14,6);
\draw (10,4) edge[multi1=1] (16,2);
\draw (13,0)--(14,6);
\draw (12,6)--(10,4)--(16,4)--(10,2)--(14,6)--(10,4);
\draw (14,6)--(16,2)--(16,4)--(12,6)--(14,6);

\end{tikzpicture}\caption{}\label{figure 14}
\end{figure}

 Now consider the graphs $G_1$ and $H_1$ in Figure \ref{figure 14}. Notice that the induced subgraphs of $G_1$ and $H_1$ on the vertex set $\{0,1,2,3,4\}$ are in both cases $G$. For the graph $G_1$ the vertices $5$ and $6$ are connected to the remaining non-root vertices with the same number of edges. But this is not the case for the graph $H_1$. Using Macaulay2 \cite{GS} we see that $\K(R/\M_{G_1}^{(1)})=\det\Q_{G_1}=28536$ but $\K(R/\M_{H_1}^{(1)})=93847>92976=\det\Q_{H_1}$. The graph $G_1$ can be written as $G*_1G'''$, where $G'''$ is the induced subgraph on the vertex set $\{0,5,6\}$. Moreover, $G'''$ is obtained from the complete multigraph $K_3^{1,1}$ by deleting some rooted edges.

In general, we propose the following conjecture on the structure of a multigraph $G$ satisfying $\K(R/\1M)=\det\Q_G$.

\begin{conjecture}\label{conjecture}
Let $G$ be a multigraph on the vertex set $\{0,1,\ldots,n\}$. Then $G$ satisfies the equality $\K(R_n/\1M)=\det\Q_G$ if and only if we can write $G=((\cdots(G_1*_{d_1}G_2)*_{d_2}\cdots *_{d_{r-2}}G_{r-1})*_{d_{r-1}}G_r)$  after a possible renumbering of the vertices. Here $G_i$ is a multigraph on the vertex set $\{0,1,\ldots,n_i\}$ obtained from the complete multigraph $K_{n_i+1}^{a_i,b_i}$ by removing edges through the root $0$ and $n=\sum_{i=1}^rn_i$. Moreover, $d_1\ge d_2\ge \cdots\ge d_{r-1}$ are nonnegative integers with $b_1\ge d_1$ and $b_i\ge d_{i-1}$ for $2\le i\le r$.
\end{conjecture}

The `if part' of the above conjecture holds true because of Corollary \ref{new graphs}. To prove the `only if' part, one possible way is to proceed as in Discussion \ref{discussion 2.9} by constructing successive maximal cliques on the set of non-root vertices and derive some results on the number of edges connecting these vertices.

\vspace{4mm}
  \noindent
  {\bf Acknowledgements:} I am grateful to Prof. Chanchal Kumar for introducing me to these problems and for his support. Many thanks to Abhay Soman and Sushil Bhunia for some helpful discussions and their suggestions. I wish to thank an anonymous referee for some helpful suggestions on an earlier version of this paper. The financial support is partially provided by CSIR, Govt. of India and DAE, Govt. of India.


\end{document}